\documentclass[onecolumn,10pt]{article}
\usepackage[top=.75in, bottom=.75in, left=.75in, right=.75in]{geometry}
\usepackage{amsmath,amsfonts,amscd,amssymb,amsthm,bbm,bm}
\usepackage{graphicx}
\usepackage{epstopdf}
\usepackage{overpic}
\usepackage{cancel}
\usepackage{rotating}
\usepackage{url}
\usepackage{caption}
\usepackage{color}
\usepackage{rotating}
\usepackage{multirow}
\usepackage{wrapfig}
\usepackage{mathtools}
\usepackage{subeqnarray}
\usepackage{setspace}
\usepackage{palatino} 
\usepackage{microtype}
\usepackage[numbers,sort&compress]{natbib}

\usepackage[bottom,flushmargin,hang,multiple]{footmisc}
\usepackage{lipsum}
\usepackage{hyperref}
\hypersetup{
  colorlinks=true,   
  linkcolor=blue,    
  citecolor=blue,    
  urlcolor=blue,     
  pdftitle={the Waring Problem of Harmonic Polynomials},  
  pdfauthor={Hualin Huang, Yilun Tang, Yu Ye, Rongmin Zhu} 
}

\definecolor{header1}{cmyk}{0,0,0,1}

\def \dim {\operatorname{dim}}

\def \R {\mathbbm{R}}

\numberwithin{equation}{section}

\theoremstyle{definition}
\newtheorem{theorem}{Theorem}[section]         
\newtheorem{definition}[theorem]{Definition}   
\newtheorem{proposition}[theorem]{Proposition} 
\newtheorem{lemma}[theorem]{Lemma}             
\newtheorem{corollary}[theorem]{Corollary}    
\newtheorem{example}[theorem]{Example}         
\newtheorem{remark}[theorem]{Remark}           
\newtheorem{algorithm}[theorem]{Algorithm}     
\newtheorem*{theorem*}{Main Theorem}

\usepackage[utf8]{inputenc}

\usepackage[normalem]{ulem}
\usepackage{color}

\usepackage{enumitem}

\title{\vspace{-.125in}{\huge\selectfont \textbf{The Waring Problem of Harmonic Polynomials}}\vspace{-.075in}}

\author{\normalsize{Hua-Lin Huang$^{1}$, Yilun Tang$^{1}$, Yu Ye $^{2}$, and Rongmin Zhu$^{1}$}\\
\footnotesize{$^1$ School of Mathematical Sciences, Huaqiao University, Quanzhou 362021, China} \\
\footnotesize{$^{2}$ School of Mathematical Sciences, University of Science and Technology of China, Hefei 230026, China} \\
\footnotesize{and} \\
\footnotesize{Hefei National Laboratory, University of Science and Technology of China, Hefei 230088, China}
}
\date{}

\begin{document}
\maketitle
\vspace{-.2in}
\begin{abstract}
This paper investigates the Waring problem of harmonic polynomials. By characterizing the annihilating ideal of a homogeneous harmonic polynomial, i.e., a real binary form that is in the kernel of the Laplacian, we show that its Waring rank equals its degree. Moreover, we show that any linear form can appear in a minimal Waring decomposition of a homogeneous harmonic polynomial, implying that the forbidden locus is  empty. We also provide an explicit algorithm for computing the minimal Waring decompositions.
\end{abstract}

\section{Introduction}
This paper studies the Waring problem of polynomials over the real field \(\mathbb{R}\). All subsequent discussion takes place in this setting unless otherwise stated. We adopt the notational conventions of \cite{Cox2015, IarrobinoKanev1999}. Let \(\mathbb{R}^* = \mathbb{R} \setminus \{0\}\) denote the set of nonzero real numbers. 

Let \(\mathcal{R} = \mathbb{R}[x,y]\) be the ring of real binary polynomials, and \(\mathcal{R}_d\) be the space of all real binary homogeneous polynomials of degree \(d\) in \(\mathcal{R}\). 
An element in \(\mathcal{R}_d\)  is called a \emph{$d$-form}, and a $1$-form is usually called a \emph{linear form}. For a $d$-form 
\(f\), a representation of as follows
\[
f = \sum_{j=1}^{r} \lambda_j L_j^d,
\]
with linear forms \(L_j\) and real numbers \(\lambda_j\), is called a \emph{Waring decomposition} of \(f\). The smallest positive integer \(r\) for which such a decomposition exists is the  \emph{Waring rank} of \(f\), denoted \(\operatorname{WR}(f)\).  A decomposition achieving this minimal \(r\) is called a \emph{minimal Waring decomposition}.

Let \(\mathcal{D} = \mathbb{R}[\partial_x, \partial_y]\) be the ring of constant-coefficient differential operators, and \(\mathcal{D}_r\) its subspace of homogeneous differential operators of degree \(r\). For any nonnegative integers \(i, j, k, l\), the action of the differential operator \(\partial_x^i\partial_y^j \in \mathcal{D}\) on the monomial \(x^k y^l \in \mathcal{R}\) is defined as
\[
\partial_x^i\partial_y^j \circ x^k y^l =
   \begin{cases}
    \frac{k!}{(k-i)!}\frac{l!}{(l-j)!}x^{k-i}y^{l-j}, & \text{if } i \leq k \text{ and } j \leq l; \\
     0,                      & \text{otherwise}.
   \end{cases}
\] 
Let $f \in \mathcal{R}$ and $\mathfrak{g} \in \mathcal{D}$. If \(\mathfrak{g} \circ f=0,\) then \(\mathfrak{g}\) is called an \textit{ annihilating polynomial of \(f\)}. We call \(f^{\perp} := \left\{ \mathfrak{g} \in \mathcal{D} \mid \mathfrak{g}\circ f = 0  \right\}\)  \textit{the annihilating ideal of \(f\)}. 

Let \(\Delta = \partial_x^2 + \partial_y^2\) denote the Laplace operator. A polynomial \(f \in \mathcal{R}\) is called \emph{harmonic} if it satisfies \(\Delta \circ f = 0\). Apparently, a polynomial is harmonic if and only if so is each of its homogeneous components. Harmonic polynomials are ubiquitous in algebra, analysis and arithmetic, see e.g. \cite{ABR, Reznick1996}.

\subsection{Problem Statement}
The Waring problem of polynomials is a classical and fascinating research direction in algebraic geometry and invariant theory, centered on whether complex mathematical objects can be decomposed into sums of simpler ones. While the Waring problem for generic polynomials over the complex field has been resolved by the celebrated Alexander-Hirschowitz theorem \cite{Hirschowitz1995}, the case over the real field remains far from complete—even for binary forms, there is a lot to be done. Since a general solution to the Waring problem of real binary forms currently appears unattainable, it is of merit to study specific families. 

To compute the Waring decomposition of a binary form, first one typically factorizes its annihilating polynomial into a product of pairwise non-proportional linear differential operators, and then applies the Apolarity Lemma (Lemma \ref{2}) to obtain the decomposition. However, according to the Fundamental Theorem of Algebra, polynomials over the complex field always split into linear factors, whereas those over the real field may contain irreducible quadratic factors, obstructing the direct application of the Apolarity Lemma in the real setting. As a nontrivial step forward, it seems natural to start with those annihilated by an irreducible quadratic differential operators. Notably, any irreducible quadratic differential operator can be transformed into the Laplacian via an invertible linear change of variables. Therefore, this paper focuses on the Waring problem of harmonic polynomials.

\subsection{Related Work}
The waring problem of complex binary forms can be traced back to the pioneering work of Sylvester \citep{Sylvester, S2} and Gundelfinger \citep{Gundelfinger1887}. Comas and Seiguer \citep{Comas2011} proposed a concise algorithm for determining the Waring rank and border rank of any complex binary form with modern language and machinery \cite{IarrobinoKanev1999}. Two authors of the present paper provided an elementary treatment for the Waring problem of complex binary forms \cite{Huang2025}. The Waring problem of real binary forms has been systematically initiated only in recent years by Comon, Ottaviani and Reznick, etc. Typical ranks of general real binary forms were considered in \citep{Comon2012, B} and completely resolved by Blekherman \citep{Blekherman2015}. The relationship between real roots and Waring rank for real binary forms was known to Sylvester \citep{S1865} and was further revealed in \citep{Comon2012, Reznick2013, CR, BS}. The computational aspect of the problem were considered in \citep{ADCZ, BSt}. To the best knowledge of the authors, the Waring decompositions for specific real binary forms were considered only for monomials \citep{BCG} and some binomials \citep{MM}.

\subsection{Methods and Results}
We pursue a real analogue of the previous work \cite{Huang2025} on complex binary forms. This paper presents an elementary and constructive approach to the Waring problem of harmonic polynomials, using only tools from linear algebra and calculus, without recourse to algebraic geometry. We first give a self-contained proof of the real Apolarity Lemma, characterizing Waring decompositions. Next, by examining the structure of homogeneous harmonic polynomials,  we determine the two generators of their annihilating ideal, see Proposition \ref{4}. Finally, we prove the main theorem, stating that the Waring rank of a homogeneous harmonic polynomial equals its degree and any linear form can appear in a minimal decomposition. We remark that, in \cite{Bernardi2025, Carlini2017} the authors consider the Waring locus of a form which are those linear forms that can occur in its minimal decompositions. The complement of a Waring locus is called a forbidden locus. In the literature, the first and the only example of empty forbidden locus was found by Flavi \cite{Flavi2024} very recently for the form \(\left(x_1^2+x_2^2+x_3^2\right)^3\). Now the infinite family of Harmonic forms are added to the list of forms with empty forbidden loci.

The main results can be summarized as follows, see Theorem \ref{8}.

\begin{theorem*}
The Waring rank of a nonzero harmonic polynomial of degree \(d\) is exactly \(d,\) and any linear form can appear in one of its minimal Waring decompositions.
\end{theorem*}

Based on this theorem and its proof, this paper further proposes a concrete algorithm for computing the minimal Waring decompositions of homogeneous harmonic polynomials, see Algorithm \ref{Algo}. In addition, we solve the Waring problem of all forms annihilated by quadratic differential operators completely as a by-product, see Theorem \ref{aq}.

\subsection{Organization of the Paper}
The paper is structured as follows. Section \(2\) establishes the existence of Waring decompositions and provides an elementary proof of the Apolarity Lemma, setting up the equivalent conditions for Waring decompositions. Section \(3\) analyzes the structure of homogeneous harmonic polynomials, determines their annihilating ideals and Waring ranks, and presents an algorithm for computing their minimal Waring decompositions.

\section{Preliminaries}
This section recalls the existence of Waring decompositions of binary forms and the Apolarity Lemma that were explained fully in \cite{Ellison1969, Reznick2013, Huang2025, IarrobinoKanev1999}.

\subsection{Existence of Waring Decompositions}

\begin{lemma}\label{5}
Let \(\alpha_1x+\beta_1y, \ \dots, \ \alpha_{d+1}x+\beta_{d+1}y\) be \(d+1\) pairwise non-proportional linear forms. Then their \(d\)-th powers  
\[
(\alpha_1x+\beta_1y)^d, \dots, (\alpha_{d+1}x+\beta_{d+1}y)^d
\]
form a basis of \(\mathcal{R}_d\). In particular, every \(f \in \mathcal{R}_d\) can be expressed as  
\[
f = \sum_{i=1}^{d+1} \lambda_i (\alpha_i x + \beta_i y)^d
\]
for some real numbers \(\lambda_1, \dots, \lambda_{d+1}\). That is, \(f\) admits a Waring decomposition.
\end{lemma}

\begin{proof}
Consider the transition matrix from the monomial basis \(\{x^d, \ x^{d-1}y, \ \dots, \ y^d\}\) of \(\mathcal{R}_d\) to the set  
\[
\{(\alpha_1x+\beta_1y)^d, \ (\alpha_2x+\beta_2y)^d, \ \dots, \ (\alpha_{d+1}x+\beta_{d+1}y)^d\}.
\]
This matrix is of Vandermonde type. Since the linear forms are pairwise non-proportional, the corresponding points \([\alpha_i : \beta_i]\) in \(\mathbb{P}^1\) are distinct, and hence the determinant of the matrix is non-zero.  
Hence the \(d\)-th powers are linearly independent, and thus form a basis of \(\mathcal{R}_d\).
\end{proof}

\subsection{The Apolarity Lemma}
We recall the Apolarity Lemma, a fundamental result in the study of Waring decompositions. 

\begin{lemma}\label{2}
Let \( f \in \mathcal{R}_d\), and let \(\alpha_1 x+\beta_1 y,\,\cdots,\, 
\alpha_r x+\beta_r y \) be \( r \) pairwise non-proportional linear forms, 
where \(1\leq r\leq d+1.\) Then the following two statements are equivalent:
\begin{enumerate}[label={(\roman*)}, font=\normalfont]
\item\label{enum:item1}
There exist nonzero real numbers \(\lambda_1,\,\cdots,\,\lambda_r\) such that 
\[
f = \sum_{i=1}^r \lambda_i ( \alpha_ix + \beta_i y)^d.
\]
\item\label{enum:item2} 
\( \prod_{i=1}^{r} \left(\beta_i\partial_x-\alpha_i\partial_y\right) \circ f = 0 \)
and \( \prod_{i \neq k} \left(\beta_i\partial_x-\alpha_i\partial_y\right) \circ f \neq  0\) for all \( 1\leq k \leq r \).
\end{enumerate}
\end{lemma}

\begin{proof}
\ref{enum:item1} $\Rightarrow$ \ref{enum:item2}: By direct computation. 

\ref{enum:item2} $\Rightarrow$ \ref{enum:item1}:
Extend the given \( r\) linear forms to a set of $d+1$ pairwise non-proportional ones:
\[ \alpha_1 x+\beta_1 y,\,\cdots,\, 
\alpha_{d+1} x+\beta_{d+1} y. \] 
By Lemma \ref{5}, there exist real numbers \(\lambda_1,\,\cdots,\,\lambda_{d+1} \) 
such that \[f = \sum_{j=1}^{d+1} \lambda_j (\alpha_jx+\beta_jy)^d.\]

Now consider the explicit action of the product \[\begin{aligned}
0=&\prod_{i=1}^{r} \left(\beta_i\partial_x-\alpha_i\partial_y\right) \circ f\\
=&\prod_{i=1}^{r} ( \beta_i\partial_x - \alpha_i\partial_y) \circ
\sum_{j=1}^{d+1} \lambda_j (\alpha_jx+\beta_jy)^d \\
=&\sum_{j=1}^{r}0+\sum_{j=r+1}^{d+1}\lambda_j\frac{d!}{(d-r)!}
\left(\prod_{i=1}^{r}\left(\beta_i\alpha_j-\alpha_i\beta_j\right)\right)(\alpha_jx + \beta_j y)^{d-r}.
\end{aligned}
\]
Due to the pairwise non-proportional assumption, we have $\beta_i\alpha_j - \alpha_i\beta_j \ne 0$ for all $i \neq j$. Again by Lemma \ref{5} with $d$ replaced by $d-r$, we deduce that \(\lambda_j=0\) for $j=r+1, \ \dots, \ d+1.\) Hence, $f$ can be expressed as \[f = \sum_{j=1}^{r} \lambda_j ( \alpha_jx + \beta_j y)^d.\]

Then by the further assumption
\[\begin{aligned}
   0 \ne &\prod_{i\neq k} \left(\beta_i\partial_x-\alpha_i\partial_y\right) \circ
\sum_{j=1}^{r} \lambda_j ( \alpha_jx + \beta_j y)^d \\
=&\sum_{j=1}^{r}\lambda_j\frac{d!}{(d-r+1)!}
\left(\prod_{i\neq k}\left(\beta_i\alpha_j-\alpha_i\beta_j\right)\right)( \alpha_jx + \beta_j y)^{d-r+1}\\
=&\lambda_k\frac{d!}{(d-r+1)!}
\left(\prod_{i\neq k}\left(\beta_i\alpha_k-\alpha_i\beta_k\right)\right)( \alpha_kx + \beta_k y)^{d-r+1}\\
\end{aligned}\]
we have $\lambda_k \neq 0$ for all $1 \le k \le r$. The proof is completed.
\end{proof}

The following is a direct corollary of the Apolarity Lemma.
\begin{corollary} \label{cal}
 Given a binary form \( f \), its Waring rank is 
\[
\operatorname{WR}(f)= \min \{ r \mid \prod_{i=1}^{r} (\beta_i \partial_x - \alpha_i \partial_y)\in f^\perp, \ \beta_i \partial_x - \alpha_i \partial_y \text{ are pairwise non-proportional} \}.
\]
\end{corollary}

\section{The Waring Problem of Harmonic Polynomials}
In this section, we investigate the Waring decompositions of Harmonic forms. We show that the Waring rank of a harmonic form is equal to its degree. Moreover, we provide an explicit, constructive algorithm for obtaining the minimal Waring decompositions. Our approach is based on analyzing their annihilating ideals.

\subsection{Properties of Harmonic Forms}

We begin with some examples of harmonic forms.
\begin{example}\label{9}
Let \[h_{d,0}=\sum_{j=0}^{\left\lfloor\frac{d}{2} \right\rfloor}(-1)^j \binom{d}{2j}x^{d-2j} y^{2j}\in\mathcal{R}_d,\]
\[h_{d,1}=\sum_{j=0}^{\left\lfloor\frac{d-1}{2} \right\rfloor}(-1)^j \binom{d}{2j+1} x^{d-2j-1} y^{2j+1}\in\mathcal{R}_d,\]where \(\left\lfloor q\right\rfloor\) denotes the greatest integer less than or equal to the real number \(q\). Computing \(\Delta \circ h_{d,0}\) yields
\[\Delta \circ h_{d,0}=d(d-1)\sum_{j=0}^{\left\lfloor\frac{d}{2}\right\rfloor-1}
\left((-1)^j+(-1)^{j+1}\right) \binom{d-2}{2j}x^{d-2-2j} y^{2j}=0.\]
Similarly, one can verify that \(\Delta \circ h_{d,1}=0.\) Therefore, \(h_{d,0}\) and \(h_{d,1}\) are harmonic forms, and so is any linear combination of \(h_{d,0}\) and \(h_{d,1}\). 
\end{example}

In fact, any harmonic \(d\)-form is a linear combination of \(h_{d,0}\) and \(h_{d,1}\), as we will prove below. 

\begin{proposition} \label{h}
Let \(f\in \mathcal{R}_d\). Then \(f\) is harmonic if and only if \(f=a_0h_{d,0}+a_1h_{d,1}.\)
\end{proposition}

\begin{proof}
One direction follows from Example \ref{9}. For the converse, suppose \(f = \sum_{j=0}^{d} \binom{d}{j} a_j x^{d-j} y^j\) is harmonic. By definition, we have \[\Delta \circ f=d!\sum_{j=0}^{d-2}\left(a_j+a_{j+2}\right)\frac{x^{d-j-2}}{(d-j-2)!}\frac{y^{j}}{j!} =0.\]
Since a polynomial vanishes identically exactly when all coefficients vanish, this implies
\begin{equation}
a_j+a_{j+2} =0 \quad (0 \leq j \leq d-2). \label{3}
\end{equation}
It follows that the coefficients of \(f\) are determined by two independent recurrence relations
\begin{enumerate} 
  \item [-] Even degree terms \( (j=2k) \): \(a_{2k}=(-1)^k a_0;\)
  \item [-] Odd degree terms \( (j=2k+1) \): \(a_{2k+1}=(-1)^k a_1.\)
\end{enumerate}
Thus \(f\) decomposes into even and odd parts:
\[\begin{aligned}
f&=a_0\sum_{j=0}^{\left\lfloor \frac{d}{2} \right\rfloor}(-1)^j \binom{d}{2j}x^{d-2j} y^{2j}
+a_1\sum_{j=0}^{\left\lfloor\frac{d-1}{2} \right\rfloor}(-1)^j \binom{d}{2j+1} x^{d-2j-1} y^{2j+1}\\
&=a_0h_{d,0}+a_1h_{d,1}.
\end{aligned}\]
Therefore, any harmonic \(d\)-form is a linear combination of \(h_{d,0}\) and \(h_{d,1}.\)
\end{proof}

\begin{remark}
    This characterization of harmonic forms can also be derived easily by the Waring decompositions of complex binary forms. As a harmonic $d$-form $f$ is annihilated by $\partial_x^2+\partial_y^2=(\partial_x + i \partial_y)(\partial_x - i \partial_y),$ it has a complex Waring decomposition as \[f=z(x+iy)^d + \overline{z}(x-iy)^d\] according to the Apolarity Lemma of complex binary forms \cite{Huang2025}. The coefficients are conjugate due to the fact that $f$ is a real polynomial. Now the property follows by a direct expansion of the right hand side of the previous equation. 
\end{remark}

\subsection{Annihilating Ideals of Harmonic forms}
We give a concrete characterization of
the annihilating ideals of Harmonic forms below.
\begin{proposition}\label{4}
Let \(f=a_0h_{d,0}+a_1h_{d,1} \ne 0\) and \(\nabla=a_1\partial_x^{d}-a_0\partial_x^{d-1}\partial_y.\) Then \(f^\perp=\left<\Delta, \nabla\right>.\)  
\end{proposition}
\begin{proof}
One can easily verify that \(\nabla\circ f = [(-1)^0-(-1)^0]a_0 a_1=0\), thus \(\langle \Delta, \nabla\rangle \subseteq f^\perp\). It remains to prove the reverse.

Take any
$\displaystyle \mathfrak{g} = \sum_{i=0}^r b_i \partial_x^{r-i} \partial_y^i \in f^\perp \cap \mathcal{D}_r.$
View $\mathfrak{g}$ and $\Delta$ as polynomials in \(\partial_y\) and conduct the division with remainder of $\mathfrak{g}$ by $\Delta$. 
By the polynomial division algorithm, see e.g. \cite{Cox2015}, there exist \( \widetilde{b}_0, \widetilde{b}_1 \in \mathbb{R} \) and \( \mathfrak{h} \in \mathcal{D}_{r-2} \) such that
\[
\mathfrak{g} = \Delta \cdot \mathfrak{h} + \widetilde{b}_0 \partial_x^r + \widetilde{b}_1 \partial_x^{r-1} \partial_y. 
\]
Since \( \mathfrak{g} \circ f = 0 \) and \( \Delta \circ f = 0 ,\) it follows that
\begin{equation}
  (\widetilde{b}_0 \partial_x^r + \widetilde{b}_1 \partial_x^{r-1} \partial_y) \circ f = 0. \label{11}
\end{equation}
For the convenience of the subsequent discussion, define
\[
\widetilde{h}_0=\frac{d!}{(d-r)!}
\sum_{j=0}^{\left\lfloor\frac{d-r}{2}\right\rfloor}(-1)^j\binom{d-r}{2j}x^{d-r-2j}y^{2j},
\]
\[\widetilde{h}_1=\frac{d!}{(d-r)!}
\sum_{j=0}^{\left\lfloor\frac{d-r-1}{2}\right\rfloor}(-1)^j\binom{d-r}{2j+1}x^{d-r-2j-1}y^{2j+1}.
\]
We now proceed by cases:
\begin{enumerate}[label=\textbf{Case \arabic*}:, leftmargin=*]
\item  \( 1 \leq r \leq d-1 \)

A direct computation of Equation \eqref{11} yields
\[
(a_0 \widetilde{b}_0 + a_1 \widetilde{b}_1) \widetilde{h}_0 + (a_1 \widetilde{b}_0 - a_0 \widetilde{b}_1) \widetilde{h}_1 = 0.
\]
Since \( \widetilde{h}_0 \) and \( \widetilde{h}_1 \) are linearly independent, one has
\[
\begin{cases}
a_0 \widetilde{b}_0 + a_1 \widetilde{b}_1 = 0, \\
a_1 \widetilde{b}_0 - a_0 \widetilde{b}_1 = 0.
\end{cases}
\]

Since \( (a_0, a_1) \neq (0,0) \), the only solution is \( \widetilde{b}_0 = \widetilde{b}_1 = 0 \). Thus,
\[
\mathfrak{g} = \Delta \cdot \mathfrak{h} \in \left< \Delta \right>.
\]

\item  \( r = d \)

In this case, Equation \eqref{11} simplifies to
\[
d! (a_0 \widetilde{b}_0 + a_1 \widetilde{b}_1) = 0.
\]
The general solution is \( \widetilde{b}_0 = \mu a_1, \widetilde{b}_1 = -\mu a_0 \) for \( \mu \in \mathbb{R} \). Therefore,
\[
\mathfrak{g} = \Delta \cdot \mathfrak{h} + \mu (a_1 \partial_x^d - a_0 \partial_x^{d-1} \partial_y) = \Delta \cdot \mathfrak{h} + \mu \nabla \in \left< \Delta, \nabla \right>.
\]
\end{enumerate}

It remains to prove that when \(r\geq d+1,\) we have \(f^\perp \cap \mathcal{D}_r\subseteq \left< \Delta,\nabla\right>.\) In this case, it is obvious that \(f^\perp\cap \mathcal{D}_r = \mathcal{D}_r.\)  Therefore, it suffices to prove that 
\(\mathcal{D}_r=\left<\Delta,\nabla\right>\cap\mathcal{D}_r\) holds.
\begin{enumerate}[label=\textbf{Case \arabic*}:, leftmargin=*, start=3]
\item  \( r = d+1 \)

Since \( \Delta \) and \( \nabla \) are coprime, it follows that \[
\dim(\left< \Delta, \nabla \right> \cap \mathcal{D}_{d+1}) = \dim(\left< \Delta \right> \cap \mathcal{D}_{d+1}) + \dim(\left< \nabla \right> \cap \mathcal{D}_{d+1}) = \dim \mathcal{D}_{d-1} + \dim \mathcal{D}_1 = d + 2 = \dim \mathcal{D}_{d+1}.
\] I.e.,  \( \mathcal{D}_{d+1} =\left< \Delta, \nabla \right> \cap \mathcal{D}_{d+1} .\) 

\item  \( r > d+1 \)

Since $\mathcal{D}_r$ can be generated by $\mathcal{D}_{d+1}=\left<\Delta,\nabla\right>\cap\mathcal{D}_{d+1},$ we have
$\mathcal{D}_r=\left<\Delta,\nabla\right>\cap\mathcal{D}_r.$
\end{enumerate}

To summarize, \( f^\perp  = \left< \Delta, \nabla \right>.\)  We are done.
\end{proof}

\subsection{The Waring Decompositions of Harmonic forms}

We are now in the position to present the main results.
\begin{theorem}\label{8}
The Waring rank of a harmonic \(d\)-form is \(d,\) and any linear form can appear in one of the minimal Waring decompositions.
\end{theorem}
\begin{proof}
 Let \(f\) be a harmonic \(d\)-form. By the Apolarity Lemma \ref{2} and Proposition \ref{4}, we have \(\operatorname{WR}(f)\geq d\). Indeed, if on the contrary \(\operatorname{WR}(f) = r < d \), then there exists an \(\mathfrak{h} \in f^\perp \cap \mathcal{D}_r\) which is a product of non-proportional linear forms due to the Apolarity Lemma \ref{2} and \(\Delta \mid \mathfrak{h}\) due to Proposition \ref{4}. This is absurd as $\Delta$ is irreducible and quadratic.

Consider an arbitrary operator 
\(\displaystyle \mathfrak{g}=\prod_{i=1}^{d-2}\left(\beta_i\partial_x-\alpha_i\partial_y\right) \in \mathcal{D}_{d-2}\)
with pairwise non-proportional linear factors. Then
\[
\Delta\circ (\mathfrak{g} \circ f) = (\Delta \cdot \mathfrak{g})\circ f = \mathfrak{g} \circ (\Delta \circ f) = 0,
\]
which shows that \(\mathfrak{g} \circ f\) is a harmonic \(2\)-form. By Proposition \ref{h}, we may write
\(
\mathfrak{g} \circ f = ax^2 + 2bxy - ay^2.
\)
Note that $(a, b) \ne (0, 0)$ since $\mathfrak{g} \circ f \ne 0$ by the fact \(\operatorname{WR}(f)\geq d\).

It remains to find two suitable linear differential operators whose product annihilates $\mathfrak{g} \circ f.$ Note that we can apply an appropriate invertible linear transformation to \(\mathfrak{g} \circ f\) so that \(\mathfrak{g} \circ f=X^2-Y^2.\) By direct verification, one has $(\delta_1 \partial_X - \gamma_1 \partial_Y)(\delta_2 \partial_X - \gamma_2 \partial_Y) \circ (X^2-Y^2) =0$ if and only if $\gamma_1\gamma_2=\delta_1\delta_2.$ Obviously, there are infinitely many choices for the $(\delta_i, \gamma_i).$ For example, take $(\delta_1, \gamma_1)=(s,t)$ and $(\delta_2, \gamma_2)=(1/s,1/t)$ for all $s,t \in \R^*.$ Therefore, we can choose $d$ pairwise non-proportional linear factors $\beta_i\partial_x-\alpha_i\partial_y$ so that \[ \prod_{i=1}^{d}\left(\beta_i\partial_x-\alpha_i\partial_y\right) \circ f =0. \] According to Corollary \ref{cal}, we have \(\operatorname{WR}(f) \le d\). 

In summary, we have proved that \(\operatorname{WR}(f) = d\). In addition, we have also showed that any linear form can appear in one of the minimal decompositions of $f.$ The corresponding minimal decomposition can be obtained by the Apolarity Lemma \ref{2}.
\end{proof}

\begin{remark}
Reznick \cite{Reznick2013} has shown that the Waring rank of a real binary form does not exceed its degree. So the claim \(\operatorname{WR}(f) = d\) of Theorem \ref{8} follows from Proposition \ref{4} and Reznick's result. Here by using intrinsic properties of Harmonic forms, we both determine the rank and explicitly construct the minimal decompositions.
\end{remark}

In addition, it was conjectured by Reznick \cite{Reznick2013} and proved by Bleckermann and Sinn \cite{BS} that a real binary $d$-form has real Waring rank $d$ if and only if it has only real roots. Combining this with our Theorem \ref{8}, one can deduce the following result which we have not been able to find in the literature.

\begin{corollary}
   Harmonic forms have only real roots.  
\end{corollary}

Based on the idea presented in the proof of Theorem \ref{8}, we propose an algorithm for constructing the minimal decompositions of Harmonic forms, using also the \(h_{d,0}, h_{d,1}\) from Example \ref{9}.

\begin{algorithm} \label{Algo}
Input a harmonic \(d\)-form \(f=a_0h_{d,0}+a_1h_{d,1}\).

\begin{enumerate}[label=\textbf{Step \arabic*}:, leftmargin=*]
  \item  Take a product \(\mathfrak{g}=\prod_{i=1}^{d-2}(\beta_i\partial_x-\alpha_i\partial_y)\) of pairwise non-proportional linear differential operators. 

  \item  Compute \(\mathfrak{g}\circ f\) and complete squares for the obtained quadratic form $ax^2+2bxy-ay^2.$

  \item  Choose a sum of squares $ax^2+2bxy-ay^2=\psi_1^2-\psi_2^2$ such that all linear forms $\alpha_1 x+\beta_1 y, \ \dots, \ \alpha_{d-2}x+\beta_{d-2}y$ together with \(\psi_1,\psi_2\) are pairwise non-proportional. 

  \item  Solve \(\lambda_1,\dots,\lambda_d\in\mathbb{R}^*\) in the equation
  \( \displaystyle
  f = \sum_{i=1}^{d-2}\lambda_i(\alpha_ix+\beta_iy)^d + \lambda_{d-1}\psi_1^d + \lambda_d\psi_2^d
  \)
  by comparing coefficients of the corresponding monomials, and output a minimal Waring decomposition of f.
\end{enumerate}
\end{algorithm}

\begin{remark}
  In Step 3 of the above algorithm, we have applied the Apolarity Lemma \ref{2}. We do not compute the annihilating polynomials of $\mathfrak{g} \circ f$ as in the proof of Theorem \ref{8}, but complete square for it instead. If $ax^2+2bxy-ay^2=(\alpha x +\beta y)^2-(\alpha' x +\beta' y)^2,$ then $(\beta \partial_x - \alpha \partial_y)(\beta' \partial_x - \alpha' \partial_y)\mathfrak{g} \circ f=0.$ Taking suitable $\alpha, \beta, \alpha', \beta'$ so that $\beta \partial_x - \alpha \partial_y, \beta' \partial_x - \alpha' \partial_y$ and the $\beta_i \partial_x - \alpha_i \partial_y$ are pairwise non-proportional, then  \( \displaystyle
  f = \sum_{i=1}^{d-2}\lambda_i(\alpha_ix+\beta_iy)^d + \lambda_{d-1}(\alpha x+ \beta y)^d + \lambda_d(\alpha' x+ \beta' y)^d
  \) according to the Apolarity Lemma \ref{2}.
\end{remark}

\subsection{Forms Annihilated by Quadratic Differential Operators}
We will digress briefly here to consider the Waring problem of all forms annihilated by quadratic differential operators. Let $0 \ne \Omega=a\partial_x^2+2b\partial_x\partial_y+c\partial_y^2 \in \mathcal{D}_2.$ It is clear that, up to a change of variables, $\Omega$ is equivalent to $\Delta$ if and only $b^2-ac<0,$ to $\pm\partial_x^2$ if and only if $b^2-ac=0,$ and to $\partial_x^2-\partial_y^2$ if and only if $b^2-ac>0.$ By combining our results in this paper and some of those in \cite{BCG, BS}, we can completely solve the Waring problem of the forms annihilated by $\Omega.$

\begin{theorem} \label{aq}
    Preserve the preceding notations and suppose $\Omega \circ f=0$ where $0 \ne f \in \mathcal{R}_d.$ Then $\operatorname{WR}(f)=d$ if and only if $b^2-ac \le 0,$ and $\operatorname{WR}(f) \le 2$ if and only if $b^2-ac > 0.$
\end{theorem}

\begin{proof}
    If $\operatorname{WR}(f)=d$, then $b^2 - ac \le 0$ according to the Apolarity Lemma \ref{2}. Otherwise, if $b^2-ac > 0$ then $\Omega$ is a product of two non-proportional linear factors, thus $\operatorname{WR}(f) \le 2$ by Corollary \ref{cal}. If $\operatorname{WR}(f) \le 2$, then $\Omega$ must have real roots again due to the Apolarity Lemma which implies $b^2-ac > 0.$

    It remains to prove that if $b^2-ac \le 0,$ then $\operatorname{WR}(f)=d$. If $b^2-ac < 0,$ then $\Omega$ is equivalent to the Laplacian, so  $\operatorname{WR}(f)=d$ by Theorem \ref{8}. If $b^2-ac=0,$ then we may assume $\Omega=\partial_x^2,$ thus $f=a xy^{d-1}+by^d=(ax+by)y^{d-1}$ which is equivalent to $xy^{d-1}.$ Then by \cite{BCG, BS}, we have $\operatorname{WR}(f)=d$.
\end{proof}

\subsection{Some Examples}
We present two concrete examples of the minimal decompositions for Harmonic forms, providing a step-by-step computation based on the above algorithm.

\begin{example}
Consider the harmonic form \(f=2h_{4,0}+h_{4,1}.\)
\begin{enumerate}[label=\textbf{Step \arabic*}:, leftmargin=*]
  \item Take
  \[
  \mathfrak{g} = (\partial_x - \partial_y)(\partial_x + \partial_y).
  \]

\item  Compute
  \[
  \mathfrak{g} \circ f = 48x^2 + 48xy - 48y^2.
  \]
  
\item  Choose
\[
\psi_1 = 2\sqrt{3}(2x + y), \quad \psi_2 = 2\sqrt{15}y,
\]  
such that  
\[
48x^2 + 48xy - 48y^2 = \psi_1^2 - \psi_2^2,
\]  
and the linear forms \( x+y,\ x-y,\ 2x+y,\ y \) are pairwise non-proportional.

\item  Assume a decomposition
  \[
  f = \lambda_1(x+y)^4 + \lambda_2(x-y)^4 + \lambda_3(2x+y)^4 + \lambda_4 y^4.
  \]
  Comparing coefficients yields
  \[
  \lambda_1 = -\tfrac{5}{2}, \quad \lambda_2 = -\tfrac{5}{6}, \quad \lambda_3 = \tfrac{1}{3}, \quad \lambda_4 = 5.
  \]
  Thus, we obtain a minimal Waring decomposition of $f$:
  \[
  f = -\frac{5}{2}(x+y)^4 - \frac{5}{6}(x-y)^4 + \frac{1}{3}(2x+y)^4 + 5y^4.
  \]
\end{enumerate}
\end{example}

\begin{example}
Consider the harmonic form \( f = h_{5,0} + 2h_{5,1}. \)
\begin{enumerate}[label=\textbf{Step \arabic*}:, leftmargin=*]
  \item Take
  \[
  \mathfrak{g} = (\partial_x - \partial_y)(\partial_x + \partial_y)(2\partial_x - \partial_y).
  \]

\item  Compute
  \[
  \mathfrak{g} \circ f = 1200xy.
  \]

\item  Choose 
\[\psi_1 = 10(x + 3y), \quad \psi_2 = 10(x - 3y), \] 
such that  
 \[ 1200xy=\psi_1^2-\psi_2^2,  \] and the linear forms \( x+y,\ x-y,\ 2x - y,\ x+3y,\ x-3y \) are pairwise non-proportional.

\item  Assume a decomposition
  \[
  f = \lambda_1(x+y)^5 + \lambda_2(x-y)^5 + \lambda_3(2x - y)^5 + \lambda_4(x-3y)^5 + \lambda_5(x+3y)^5.
  \]
  Comparing coefficients yields
  \[
  \lambda_1 = \tfrac{25}{8}, \quad \lambda_2 = -\tfrac{25}{24}, \quad \lambda_3 = -\tfrac{4}{3}, \quad \lambda_4 = \tfrac{1}{24}, \quad \lambda_5 = \tfrac{5}{24}.
  \]
  Thus, we obtain a minimal Waring decomposition of $f$:
  \[
  f = \frac{25}{8}(x+y)^5 - \frac{25}{24}(x-y)^5 - \frac{4}{3}(2x-y)^5 + \frac{1}{24}(x-3y)^5 + \frac{5}{24}(x+3y)^5.
  \]
\end{enumerate}
\end{example}

\section*{Use of AI tools declaration}
The authors declare they have not used Artificial Intelligence (AI) tools in the creation of this article.

\section*{Acknowledgements}
H.-L. Huang was partially supported by the Key Program of the Natural Science Foundation of Fujian Province (Grant No. 2024J02018) and the National Natural Science Foundation of China (Grant No. 12371037).

Y. Ye was partially supported by the National Key R\&D Program of China (Grant No. 2024YFA1013802), the National Natural Science Foundation of China (Grant Nos. 12131015 and 12371042), and the Innovation Program for Quantum Science and Technology (Grant No. 2021ZD0302902).

R. Zhu was partially supported by the National Natural Science Foundation of China (Grant No. 12201223) and the Natural Science Foundation of Fujian Province (Grant Nos. 2023J05048 and 2023J01126).
\section*{Conflict of interest}
The authors declare no conflicts of interest.

\begin{spacing}{.88}
\setlength{\bibsep}{2.pt}
\bibliographystyle{abbrvnat}

\begin{thebibliography}{99}

\bibitem{Hirschowitz1995} J. Alexander, A. Hirschowitz.
\emph{Polynomial interpolation in several variables.}
J. Algebraic Geom. 4(4) (1995) 201-222.

\bibitem{ADCZ} M. Ansola, A. Díaz-Cano, M.A. Zurro.
\emph{Semialgebraic sets and real binary forms decompositions.}
J. Symb. Comput. 107 (2021) 209–220.

\bibitem{ABR} S. Axler, P. S. Bourdon, W. C. Ramey.
\emph{Harmonic function theory. Second edition.}
Grad. Texts in Math. 137, 
Springer-Verlag, New York, 2001.

\bibitem{B} E. Ballico. 
\emph{On the typical rank of real bivariate polynomials.}
Linear Algebra Appl. 452 (2014) 263–269.

\bibitem{Bernardi2025} A. Bernardi, A. Oneto, P. Santarsiero.
\emph{Decomposition loci of tensors.}
J. Symb. Comput. 131 (2025) 102451.


\bibitem{Blekherman2015} G. Blekherman.
\emph{Typical real ranks of binary forms.}
Found. Comput. Math. 15(3) (2025) 793-798.


\bibitem{BS} G. Blekherman, R. Sinn. 
\emph{Real rank with respect to varieties.}
Linear Algebra Appl. 505 (2016) 344–360.

\bibitem{BCG} M. Boij, E. Carlini, A.V. Geramita.
\emph{Monomials as sums of powers: the real binary case.} 
Proc. Am. Math. Soc. 139 (2011) 3039–3043.

\bibitem{BSt} M. C. Brambilla, G. Staglian\'o.
\emph{Algebraic boundaries among typical ranks for real binary forms of arbitrary degree.} 
Found. Comput. Math. 21(4) (2021) 1003-1022. 

\bibitem{Carlini2017} E. Carlini, M. V. Catalisano, A. Oneto.
\emph{Waring loci and the Strassen conjecture.}
Adv. Math. 314 (2017) 630-662.

\bibitem{CR} A. Causa, R. Re.
\emph{On the maximum rank of a real binary form.}
Ann. Mat. Pura Appl. 190 (1) (2011) 55–59.

\bibitem{Comas2011} G. Comas, M. Seiguer.
\emph{On the rank of a binary form.}
Found. Comput. Math. 11(1) (2011) 65-78.

\bibitem{Comon2012} P. Comon, G. M. Ottaviani.
\emph{On the typical rank of real binary forms.}
Linear Multilinear Algebra 60(6) (2012) 657-667.

\bibitem{Cox2015} D. A. Cox, J. Little, D. O'Shea.
\emph{Ideals, varieties, and algorithms. Fourth edition. }
Undergraduate Texts in Mathematics, Springer, Cham, 2015.

\bibitem{Ellison1969} W. Ellison.
\emph{A `Waring's problem' for homogeneous forms.}
Math. Proc. Cambridge Philos. Soc. 65(3) (1969) 663-672.

\bibitem{Flavi2024} C. Flavi.
\emph{Decompositions of powers of quadrics.}
Dissertationes Math. 602 (2025), 108 pp.

\bibitem{Gundelfinger1887} S. Gundelfinger.
\emph{Zur Theorie der binären Formen.}
J. Reine Angew. Math. 100 (1887) 413-424.

\bibitem{Huang2025} H.-L. Huang, H. Miao, Y. Ye.
\emph{The Waring problem of complex binary forms.}
arXiv: 2511.14316.

\bibitem{IarrobinoKanev1999} A. Iarrobino, V. Kanev.
\emph{Power Sums, Gorenstein Algebras, and Determinantal Loci.}
Lect. Notes in Math., vol. 1721. Springer, Berlin, 1999.

\bibitem{MM} L. B. Moncus\'i, S. K. Masuti.
\emph{The Waring rank of binary binomial forms.}
Pacific J.   Math. 313(2) (2021) 327-342.



\bibitem{Reznick1996} B. Reznick.
\emph{Homogeneous polynomial solutions to constant coefficient PDE's.}
Adv. Math. 117 (1996) 179–192. 


\bibitem{Reznick2013} B. Reznick. \emph{On the length of binary forms.} in
Quadratic and Higher Degree Forms, Dev. Math. 31, Springer,  New York, 2013, 207–232.  

\bibitem{Sylvester} J. J. Sylvester.
\emph{An essay on canonical forms, supplement to a sketch of a memoir on elimination, transformation and canonical forms.}
in: H.F. Baker (Ed.), The collected mathematical papers of James Joseph Sylvester. Vol. 1, (1837-1853), Cambridge University Press, Cambridge, 1904, 203-216. 
	
\bibitem{S2} J. J. Sylvester.
\emph{On a remarkable discovery in the theory of canonical forms and of hyperdeterminants.} 
in: H.F. Baker (Ed.), The collected mathematical papers of James Joseph Sylvester. Vol. 1, (1837-1853), Cambridge University Press, Cambridge, 1904, 265-283. 

\bibitem{S1865} J. J. Sylvester.
\emph{Syllabus of Lecture Delivered at King’s College, London, June 28,
1865: Elementary Proof and Generalization of Sir Isaac Newton’s Hitherto Undemonstrated Rule for
the Discovery of Imaginary Roots.} 
Proc. Lond. Math. Soc. 1 (1865) 1–12.

\end{thebibliography}
\end{spacing}

\end{document}